\documentclass[psamsfonts]{amsart}

\usepackage{amssymb,amsfonts}
\usepackage[all,arc]{xy}
\usepackage{enumerate}
\usepackage{mathrsfs}
\usepackage{graphicx}

\usepackage{orcidlink}

\newtheorem{theorem}{Theorem}[section]
\newtheorem{corollary}[theorem]{Corollary}
\newtheorem{proposition}[theorem]{Proposition}
\newtheorem{lemma}[theorem]{Lemma}

\theoremstyle{definition}
\newtheorem{definition}[theorem]{Definition}

\newtheorem{example}[theorem]{Example}

\newtheorem{question}[theorem]{Question}
\newtheorem{remark}[theorem]{Remark}

\DeclareMathOperator{\Ass}{Ass}
\DeclareMathOperator{\Ann}{Ann}

\DeclareMathOperator{\Id}{Id}

\DeclareMathOperator{\Spec}{Spec}

\DeclareMathOperator{\End}{End}


\numberwithin{equation}{section}

\begin{document}

\title{Covering conditions for ideals in semirings}

\author{Peyman Nasehpour}

\email{nasehpour@gmail.com}

\address{Peyman Nasehpour\orcidlink{0000-0001-6625-364X}\\ Academic Advisor and Education Mentor \\ Education Department \\ The New York Academy of Sciences}



\begin{abstract}
In this paper, we prove prime avoidance for ringoids. We also generalize McCoy's and Davis' prime avoidance theorems in the context of semiring theory. Next, we proceed to define and characterize compactly packed semirings and show that a commutative semiring is compactly packed if and only if each prime ideal is the radical of a principal ideal. Finally, we calculate the set of zero-divisors of some monoid semimodules over compactly packed semirings in terms of their prime ideals.
\end{abstract}

\maketitle

\section{Introduction}

Prime avoidance is a fundamental result in commutative ring theory, discussed in many commutative algebra textbooks including \cite{BrunsHerzog1998,Eisenbud1995,Kaplansky1974,Kemper2011,Matsumura1989}, which has applications in algebraic geometry (cf. \cite[p. 200]{BrunsGubeladze2009}, \cite[p. 204]{Eisenbud2005}, \cite[p. 555]{GoertzWedhorn2020}, \cite[p. 44]{Schenck2003}, and \cite{Vakil2025}), algebraic number theory (\cite[p. 65]{Ash2010}), group schemes of finite type (cf. \cite[p. 597]{Milne2017}) and homological algebra (cf. \cite[p. 145]{BrodmannSharp2012} and \cite[p. 490]{Rotman2009}).

In his celebrated paper, McCoy \cite[Theorem 1]{McCoy1957} proved that if $I$ and $\{A_i\}_{i=1}^n$ are ideals of a commutative ring $R$ with $I \subseteq \bigcup_{i=1}^n A_i$ such that $I$ is not contained in the union of any $n-1$ of ideals $A_i$, then there is a positive integer $k$ such that $I^k \subseteq \bigcap_{i=1}^n A_i$. A corollary to this beautiful result is a covering condition for radical ideals in this sense that if an ideal $I$ is contained in a union of finitely many radical ideals $\{A_i\}_{i=1}^{n}$, then $I$ is contained in at least one of them \cite[p. 163-4]{McAdam1974}. What is known as ``prime avoidance'' in commutative algebra is, in fact, a corollary to the latter result, and this is why Kaplansky attributes it to McCoy. However, the only book that the author found discussing prime avoidance in the context of noncommutative algebra is Rowen's work (see Proposition 2.12.7 in \cite{Rowen1991}). The main purpose of this paper is to investigate these covering conditions in ringoid theory, especially within the framework of semiring theory. Since the language of ringoid theory is not standardized yet, we need to establish some terminology.

Let us recall that a set $M$ with a map $M \times M \rightarrow M$ defined by $(x,y) \mapsto xy$ is a magma \cite[Definition 1.1]{Serre1992}. A bimagma $(R,+,\cdot)$ is a ringoid \cite[p. 206]{Rosenfeld1968} if multiplication ``$\cdot$'' distributes on addition ``$+$'' from both sides, i.e., for all $r$, $s$, and $t$ in $R$, \[r(s+t) = rs + rt \text{~and~} (s+t)r = sr + tr.\] If $(R,+,\cdot)$ is a ringoid then $(R,+)$ and $(R,\cdot)$ are the additive and multiplicative magmas of $R$, respectively. Ringoids are the most general among all ring-like algebraic structures although different from the so-called ringoids investigated in \cite{Sehgal1964}. In \S\ref{sec:ringoids}, we discuss ringoids and provide some general examples. For instance, in Proposition \ref{Endomorphismringoidmedialmagma}, we show that if $(M,+)$ is a medial magma, then $\End(M)$ equipped with componentwise addition and composition of functions is a ringoid. Recall that a magma $(M,+)$ is said to be medial \cite{FateloMartins2016} if \[(a+b)+(c+d) = (a+c)+(b+d), \qquad\forall~a,b,c,d \in M.\] Also, note that if $M$ is a magma, we denote the set of all magma endomorphisms of $M$ by $\End(M)$. In this section, we also introduce nonassociative hemirings that are examples of ringoids with zero. Similar to the theory of general algebras over fields, we construct nonassociative hemialgebras over semifields (see Theorem \ref{finitedimensionalhemiringoversemifield}). We recall that a commutative semiring is a semifield if each ot its nonzero elements is multiplicatively invertible (see p. 52 in \cite{Golan1999(b)}). Other examples of ringoids with zero include Newman algebras (check Theorem \ref{newmanalgebrasnasemirings}).

In \S\ref{sec:idealsofringoids}, we discuss ideals of ringoids. Similar to semiring theory, we say a nonempty subset $I$ of a ringoid $R$ is a left (right) ideal of $R$ if $(I,+)$ is a submagma of $(R,+)$ and $ra \in I$ ($ar \in I$), for all $r \in R$ and $a \in I$. A nonempty subset $I$ of a ringoid $R$ is an ideal of $R$ if it is both a left and a right ideal of $R$. We collect all left (right) ideals of a ringoid $R$ in $\Id_l(R)$ ($\Id_r(R)$). Also, we collect all ideals of a ringoid $R$ in $\Id(R)$.

A left (right) ideal $I$ of a ringoid $R$ is subtractive if $x+y \in I$ and $x\in I$ imply that $y\in I$, and $x+y \in I$ and $y\in I$ imply that $x\in I$, for all $x,y\in R$ (see Definition \ref{subtractiveidealsdef}). A ringoid is subtractive if each of its ideals is subtractive. In Theorem \ref{Avoidanceofsubtractiveideals}, we show that if $A$, $B$, and $C$ are left (right) ideals of a ringoid $(R,+,\cdot)$ such that $C \subseteq A \cup B$ and $A$ and $B$ are subtractive, then either $C \subseteq A$ or $C \subseteq B$.

A left (right) ideal $I$ of a ringoid $R$ is proper if $I \neq R$. Inspired by the definition of prime ideals in nonassociative ring theory (see Definition 1 in \cite{Behrens1956}), we say a proper ideal $P$ of a ringoid $R$ is prime if $(a)(b) \subseteq P$ implies either $a \in P$ or $b \in P$, for all $a$ and $b$ in $R$ (see Definition \ref{primeidealringoiddef}). We collect all prime ideals of a ringoid $R$ in $\Spec(R)$.

In Theorem \ref{primeidealringoidthm}, we prove that a proper ideal $P$ of a ringoid $R$ is prime if and only if $IJ \subseteq P$ implies either $I \subseteq P$ or $J \subseteq P$, for all ideals $I$ and $J$ of $R$. In Theorem \ref{Behrenlemma}, we verify that if $\{P_i\}_{i=1}^{n}$ is a family of prime ideals of a ringoid $R$ and $I$ is an ideal of $R$ such that for each $i \in \mathbb{N}_n$ there is an element $a_i \in I$ with $a_i \in P_i$ but $a_i \notin P_k$ for each $k \neq i$, then for each $l \in \mathbb{N}_n$, there is an element $b_l$ belonging to $I$ and each $P_i$ except $P_l$. We use this, to generalize prime avoidance for ringoids in the following sense:

Let $P_1, \dots, P_n$ be subtractive prime ideals of a ringoid $(R,+,\cdot)$. Let $I$ be an ideal of $R$ with $I \nsubseteq P_i$, for each $i \in \mathbb{N}_n$. Then in Theorem \ref{PALforringoids}, we prove that there is an element $a \in I$ avoiding to be in any of the $P_i$s, i.e., $a \in I \setminus \bigcup_{i=1}^{n} P_i$. Note that this is a generalization of Behrens' prime avoidance for nonassociative rings \cite[Satz 3]{Behrens1956}.

Semirings, which have recently gained the attention of many algebraists and computer scientists, are interesting generalizations of rings and bounded distributive lattices. They also have essential applications across various fields in science and engineering \cite{Glazek2002,Golan1999(b),GondranMinoux1984,HebischWeinert1998}. In this paper, a ringoid $(S,+,\cdot)$ is a semiring if $(S,+,0)$ is a commutative monoid, $(S,\cdot,1)$ is a monoid, and 0 $(\neq$ 1) is an absorbing element of $(S,\cdot)$, i.e., $s \cdot 0 = 0 = 0 \cdot s$, for all $s \in S$. If all above conditions are satisfied, but the multiplication is not necessarily associative, we say that $S$ is a nonassociative semiring. Here is a good place to recall the definition of semimodules also. Let $S$ be a semiring. A commutative monoid $(M,+,0)$ is, by definition, an $S$-semimodule \cite[\S14]{Golan1999(b)} if there exists a scalar multiplication function \[\lambda: (S,M) \rightarrow M \quad \text{defined~by} \quad \lambda(s,m) = sm,\] with the following properties for all $s,t \in S$ and $m,n \in M$:

\begin{enumerate}
	\item $(st)m = s(tm)$ and $1m = m$,
	\item $(s+t)m = sm + tm$ and $s (m+n) = sm + sn$,
	\item $0 m = 0$ and $s 0 = 0$.
\end{enumerate}

In \S\ref{sec:pat}, we prove a generalization of prime avoidance in another direction. In Theorem \ref{PALforsemirings}, as a generalization to prime avoidance in noncommutatve algebra \cite[Proposition 2.12.7]{Rowen1991}, we prove that if an ideal of a semiring $S$ is contained in a union of $n$ subtractive ideals of $S$ and at least $n-2$ of them are prime, then it is contained in some of them. Note that in Example \ref{nosubtractivenoPAL}, we construct a general example to show that the prime avoidance for semirings may not hold if the prime ideals $\{P_i\}_{i=1}^{n}$ in Theorem \ref{PALforsemirings} are not subtractive.

As an application of Theorem \ref{PALforsemirings}, we prove that if $S$ is a semiring and satisfies a.c.c. on $M$-annihilator ideals of $S$, where $M$ is an $S$-semimodule, and also, if $I$ is an ideal of $S$ and a subset of a finite union of $M$-annihilator ideals of $S$, then $I$ is contained in an $M$-annihilator and prime ideal of $S$ (see Theorem \ref{containedannihilatorcontainedprime}). Note that by definition, an ideal $I$ of $S$ is an $M$-annihilator ideal if $I = \Ann(X)$, for some nonempty $X \subseteq M$, where $M$ is an $S$-semimodule. Also, recall that similar to module theory, an $S$-semimodule $M$ over a semiring $S$ has a.c.c. (d.c.c.) on its $M$-annihilator ideals if any ascending (descending) chain of $M$-annihilator ideals of $S$ terminates at some point. 

We also add that the topological interpretation of prime avoidance for semirings is that if $S$ is a subtractive semiring and a finite number of points are contained in an open subset then they are contained in a smaller principal open subset (check Theorem \ref{primeavoidancetopology}).

Next in this section, we proceed to prove Davis' version for prime avoidance in the context of semiring theory (check Theorem \ref{Davisprimeavoidancelemma}). Davis' version of prime avoidance in commutative ring theory has some applications in grades of ideals (see Theorem 124 and Theorem 125 in \cite{Kaplansky1974}).  

In \S\ref{sec:finiteunions}, we generalize McCoy's results on finite unions of ideals in commutative semirings. Note that a semiring $S$ is commutative if $ab = ba$, for all $a$ and $b$ in $S$. In Theorem \ref{finiteintersectionmccoythm}, we show that if $n \geq 3$ is a positive integer, and $I$ and $\{A_i\}_{i=1}^n$ are ideals of a subtractive commutative semiring $S$ such that the covering $I \subseteq  \bigcup_{i=1}^n A_i$ is efficient, then there is a positive integer $k$ such that $I^k \subseteq  \bigcap_{i=1}^n A_i$. Two corollaries to this result are two covering conditions which we call ``McAdam's radical ideal avoidance for semirings'' (see Corollary \ref{mcadamradicalidealavoidancelemma}) and ``McCoy's semiprime avoidance for semirings'' (see Corollary \ref{mccoysemiprimeavoidancelemma}). Recall that a covering $I \subseteq \bigcup_{i=1}^{n} A_i$ is efficient if $I$ is not contained in the union of any $n-1$ of $A_i$s \cite[\S2]{McAdam1974}. At the end of \S\ref{sec:finiteunions}, we use McCoy's semiprime avoidance for semirings to show that if $T$ is a multiplicatively closed subset of a subtractive commutative semiring $S$ and $\{P_i\}_{i=1}^{n}$ are $T$-semiprime and $2$-absorbing ideals of $S$ and $I$ is an ideal of $S$ such that $I \subseteq \bigcup_{i=1}^{n} P_i$, then $tI$ is contained in one the $P_i$s, for some $t \in T$ (see Theorem \ref{semiprimeavoidance2absorbing}). Note that if $S$ is a commutative semiring, $T$ a multiplicatively closed subset of $S$, and $P$ an ideal of $S$ with $P \cap T = \emptyset$, then we say $P$ is a $T$-semiprime ideal of $S$ if there is a $t \in T$ such that $s^2 \in P$ implies $ts \in P$, for all $s \in S$ (check Definition \ref{mcssemiprimedef}). We recall that a proper ideal $P$ of a semiring $S$ is semiprime if for any ideal $I$ of $S$, $I^2 \subseteq P$ implies $I \subseteq P$ (see p. 90 in \cite{Golan1999(b)}). Also, note that a proper ideal $P$ of a commutative semiring $S$ is semiprime if and only if $s^2 \in P$ implies $s \in P$, for all $s \in S$. We also recall that a proper ideal $I$ of a commutative semiring $S$ is a $2$-absorbing ideal of $S$ if $xyz\in I$ implies either $xy\in I$, or $yz \in I$, or $xz\in I$ (cf. Definition 2.1 in \cite{Darani2012}). For more on $2$-absorbing ideals of semirings, refer to \cite{BehzadipourNasehpour2020}.

We devote \S\ref{sec:compactlypackedsemirings} to ``arbitrary prime avoidance property'' (see \cite{Chen2021} and Definition 2.2 in \cite{Nasehpour2025}) which is defined in the context of semiring theory as follows:

\begin{itemize}
	\item For an arbitrary family of prime ideals $\{P_\alpha\}$ and any ideal $I$ of a semiring $S$, the inclusion $I \subseteq  \bigcup_{\alpha} P_\alpha$ implies $I \subseteq P_\alpha$, for some $\alpha$.
\end{itemize}

 The ring version of this property has been investigated in \cite{PakalaShores1981,ReisViswanathan1970,Smith1971,UregenTekirOral2016}. In Definition \ref{compactlypackedsemirings}, we call semirings with this covering condition ``compactly packed'', and next, we prove (see Theorem \ref{generalizedpal}) the following statements are equivalent for a commutative semiring $S$:

\begin{enumerate}
	\item The semiring $S$ is compactly packed.
	
	\item For an arbitrary family of prime ideals $\{P_\alpha\}$ and any prime ideal $Q$ of $S$, the inclusion $Q \subseteq  \bigcup_{\alpha} P_\alpha$ implies $Q \subseteq P_\alpha$, for some $\alpha$.
	
	\item Each prime ideal of $S$ is the radical of a principal ideal in $S$.
	
	\item Each radical ideal is the radical of a principal ideal.	
\end{enumerate}

A semiring is considered proper if it is not a ring \cite[p. 9]{HebischWeinert1998}. Our generalizations will be more meaningful if we can provide some compactly packed proper semirings. This task is accomplished in Example \ref{examplescompactlypacked}, Proposition \ref{examplecompactlypacked2}, and Proposition \ref{examplecompactlypacked3}.

In \S\ref{sec:zerodivisors}, we discuss zero-divisors on semimodules. Recall that an element $s$ in a commutative semiring $S$ is a zero-divisor on the $S$-semimodule $M$ if there is a nonzero element $m$ in $M$ such that $sm = 0$. All zero-divisors on $M$ are collected in $Z(M)$. In Proposition 1.15 in \cite{AtiyahMacdonald2016}, it is proved that the set of zero-divisors of a commutative ring is a union of radicals of annihilator ideals. As a generalization to this result, we show that the set of zero-divisors of a semimodule over a commutative semiring is a union of radicals of $M$-annihilator ideals of $S$ (see Proposition \ref{zerodivisorssemimoduleunionradicalideals}). Recall that a semimodule $M$ over a commutative semiring $S$ has ``very few zero-divisors'' if the set of zero-divisors $Z(M)$ of $M$ is a finite union of prime ideals in $\Ass(M)$ (see Definition 2.7 in \cite{Nasehpour2021}). As a generalization of Corollary 2.9 in \cite{Nasehpour2021}, we show that if $S$ has a.c.c. and d.c.c. on its $M$-annihilator ideals, then $M$ has very few zero-divisors (see Theorem \ref{accdccannidealsvfzd}).

Recall that a prime ideal $P$ of a commutative semiring $S$ is an associated prime ideal of an $S$-semimodule $M$ if there is an $m \in M$ such that \[P = \Ann(m) = \{s \in S: sm = 0\}.\] All associated prime ideals of a semimodule $M$ are collected in $\Ass(M)$ \cite{Nasehpour2021}. An $S$-semimodule $M$ has Property (A) if each finitely generated ideal $I$ of $S$ with $I \subseteq Z(M)$ has a nonzero annihilator in $M$ (Definition 2.14 in \cite{Nasehpour2021}). In Theorem \ref{amongmostusefulkaplansky}, we use semiring version of prime avoidance to prove that if $I$ is an ideal of a commutative semiring $S$ and $M$ a Noetherian $S$-semimodule, then $I \subseteq Z(M)$ implies that $I \subseteq P$, for some $P \in \Ass(M)$. This fact that a Noetherian module over a Noetherian commutative ring has Property (A) is considered ``among the most useful'' results in commutative algebra by Kaplansky \cite[p. 56]{Kaplansky1974}.

In \S\ref{sec:zerodivisors}, we proceed to generalize another result of Davis and prove that if a commutative semiring $S$ has few zero-divisors, then its total quotient semiring $Q(S)$ is semi-local (check Theorem \ref{fzdtquotientsemiringsemilocal}). Recall that a commutative semiring is semi-local if it has finitely many maximal ideals. Moreover, in Definition \ref{semiringfewzerodivisors}, we say a commutative semiring $S$ has few zero-divisors if $Z(S)$ is a finite union of subtractive prime ideals of $S$. A corollary to Theorem \ref{fzdtquotientsemiringsemilocal} is that if a commutative semiring $S$ has a.c.c. on its annihilator ideals, then $Q(S)$ is a semi-local Kasch semiring (see Corollary \ref{acctqsemilocalKasch}). We add that we define a commutative semiring $S$ to be a Kasch semiring if each maximal ideal of $S$ is of the form $\Ann(x)$, for some $x \in S$ (see Definition \ref{Kaschsemiringdef}).

In the final phase of \S\ref{sec:zerodivisors}, we prove that if $S$ is a compactly packed commutative semiring, $M$ an $S$-semimodule with Property (A) and $Z(M) = \bigcup_{\alpha \in A} P_\alpha$, where $P_\alpha$s are prime ideals of $S$, and $G$ is a cancellative torsion-free commutative monoid, then $Z(M[G]) =  \bigcup_{\alpha \in A} P_\alpha[G]$ (see Theorem \ref{zerodivisorsmonoidsemimodule1}). Similarly in Theorem \ref{zerodivisorsmonoidsemimodule2}, we show that if $S$ is a compactly packed Noetherian commutative semiring, $M$ an $S$-semimodule, and $G$ is a cancellative torsion-free commutative monoid, then $Z(M[G]) = \bigcup_{P \in \Ass(M)} P[G]$.

Golan's book \cite{Golan1999(b)} is a general reference for semiring theory, and our terminology closely follows it.

\section{Ringoids and some examples}\label{sec:ringoids}

A bimagma $(R,+,\cdot)$ is a ringoid \cite[p. 206]{Rosenfeld1968} if multiplication ``$\cdot$'' distributes on addition ``$+$'' from both sides. We say a ringoid $R$ is with zero if there is an element $0$ in $R$ such that $0$ is a neutral element for the magma $(R,+)$ and an absorbing element for the magma $(R,\cdot)$. In the following, we give a general example of a ringoid:

\begin{proposition}\label{Endomorphismringoidmedialmagma}
	Let $(M,+)$ be a medial magma. Then, $\End(M)$ equipped with componentwise addition and composition of functions is a ringoid such that $\End(M)$ equipped with $+$ is medial and equipped with $\circ$ a monoid.
\end{proposition}

\begin{proof}
Since the magma $M$ is medial, for $f$ and $g$ in $\End(M)$, we have \[(f(x) + f(y)) + (g(x) + g(y)) = (f(x) + g(x)) + (f(y) + g(y)).\] Applying this, one can easily prove that $(\End(M),+)$ is a medial magma. It is routine to see that \[(\End(M),\circ,\iota_M)\] is a monoid. The right distributive law can be checked easily. Now, we prove that $\circ$ distributes on $+$ from the left side. Let $f$, $g$, and $h$ be endomorphisms of $M$ and consider an arbitrary element $m$ of $M$. Observe that \begin{align*}
		(f \circ (g+h))(m) & = f((g+h)(m)) = f(g(m) + h(m)) \\ & = f(g(m)) + f(h(m)) = (f \circ g + f \circ h)(m). 
	\end{align*} This completes the proof.
\end{proof}

Let $(M,+,0)$ be a unital magma and $\End_0(M)$ the set of magma endomorphisms $f$ with $f(0) = 0$. In the following, we give a general example for ringoids with zero:

\begin{proposition}
Let $(M,+,0)$ be a medial and unital magma. Then, $\End_0(M)$ equipped with $+$ and $\circ$ is a ringoid with zero.
\end{proposition}

\begin{proof}
The zero of the ringoid $\End_0(M)$ is the zero function defined by $0(m) = 0$, for all $m \in M$.
\end{proof}

Now, we proceed to give more examples for ringoids with zero. Examples of ringoids with zero include neorings introduced in \S2 of Bruck's paper \cite{Bruck1955}. However, by having a family of ringoids with zero, it is possible to contruct new ones from the old ones. For instance, let $\{R_\alpha\}_\alpha$ be a family of ringoids. Define addition and multiplication on the direct product $\prod_{\alpha} R_\alpha$ of the given ringoids componentwisely. Additionally, let each ringoid $R_\alpha$ be with zero. Define their direct sum $\bigoplus_{\alpha} R_\alpha$ consisting those elements in $\prod_{\alpha} R_\alpha$ whose components are zero except finitely many of them. It is straightforward to see the following: 

\begin{proposition}
If $\{R_\alpha\}_\alpha$ is a family of ringoids, then their direct product $\prod_{\alpha} R_\alpha$ equipped with componentwise addition and multiplication is also a ringoid. On the other hand, if each $R_\alpha$ is with zero, then $\prod_{\alpha} R_\alpha$ is with zero and $\bigoplus_{\alpha} R_\alpha$ is a subringoid of $\prod_{\alpha} R_\alpha$.
\end{proposition}

\begin{definition}
We say a ringoid $H$ with zero is an NA-hemiring (i.e., a nonassociative hemiring) if $(H,+,0)$ is a commutative monoid and $0$ is an absorbing element for the magma $(H,\cdot)$.
\end{definition}
 
Similar to the theory of finite dimensional algebras over fields \cite{DrozdKirichenko1994}, we construct nonassociative semialgebras over semifields. To achieve this, let $K$ be a semifield and consider the $K$-semimodule $H = K^n$, where $n \in \mathbb{N}$. Let $\{a_i\}_{i=1}^{n}$ be the standard basis for $H$, i.e., the family of $n$ elements in $H$, where the $i$th coordinate of $a_i \in H$ is $1$ and the other coordinates are $0$ for each $i \in \mathbb{N}_n$. In order to define a multiplication $\cdot$ on $H$, first we choose $n^2$ elements $b_{ij}$ in $H$ and set \[a_i a_j = b_{ij}, \qquad\forall~i,j \in \mathbb{N}_n.\] Each $b_{ij}$ is determined by the $n$ elements $\{\gamma^k_{ij}\}_{k=1}^{n} \subseteq K$ with $b_{ij} = \sum_{k=1}^{n} \gamma^k_{ij} a_k$. Now, if $a$ and $b$ are arbitrary elements of $H$, we have $a = \sum_{i=1}^{n} \alpha_i a_i$ and $b = \sum_{j=1}^{n} \beta_j a_j$, where the coefficients $\{\alpha_i\}_{i=1}^n$ and $\{\beta_j\}_{j=1}^n$ are uniquely determined. Now, we define the multiplication of $a$ and $b$ as follows:

\begin{align} \left( \sum_{i=1}^{n} \alpha_i a_i \right) \left( \sum_{j=1}^{n} \beta_j a_j \right) = \sum_{i,j = 1}^{n} \alpha_i \beta_j b_{ij} = \sum_{i,j,k = 1}^{n} \alpha_i \beta_j \gamma^k_{ij} a_k .\label{semialgebramultiplication} \end{align}

\begin{theorem}\label{finitedimensionalhemiringoversemifield}
	Let $K$ be a semifield. On $H = K^n$, define addition componentwisely and multiplication as in (\ref{semialgebramultiplication}). Then, $H$ is a nonassociative hemiring with the following property: \[(\alpha a)b = a (\alpha b) = \alpha (ab), \qquad\forall~\alpha \in K,~\forall~a,b \in H.\]
\end{theorem}

\begin{proof}
	It is clear that $(H,+,\mathbf{0})$ is a commutative monoid and $(H,\cdot)$ a magma. On the other hand, since each coordinate of $\mathbf{0} \in H$ is $0 \in K$, by the definition of multiplication in (\ref{semialgebramultiplication}), $\mathbf{0}$ is an absorbing element of the magma $(H,\cdot)$. The proof of the property \[(\alpha a)b = a (\alpha b) = \alpha (ab), \qquad\forall~\alpha \in K,~\forall~a,b \in H\] is routine and omitted. Now, let $a = \sum_{i=1}^{n} \alpha_i a_i$, $b = \sum_{j=1}^{n} \beta_j a_j$, and $c = \sum_{k=1}^{n} \gamma_k a_k$ be arbitrary elements of $H$. Observe that \[b+c = \sum_{j=1}^{n} (\beta_j + \gamma_j) a_j \text{~and~}\] \begin{align*}
		a(b+c) & = \left( \sum_{i=1}^{n} \alpha_i a_i\right) \left( \sum_{j=1}^{n} (\beta_j + \gamma_j) a_j \right) \\ & = \sum_{i,j,k = 1}^{n} \alpha_i (\beta_j + \gamma_j) \gamma^k_{ij} a_k \\ & = \sum_{i,j,k = 1}^{n} \alpha_i \beta_j \gamma^k_{ij} a_k + \sum_{i,j,k = 1}^{n} \alpha_i \gamma_j \gamma^k_{ij} a_k = ab + ac.
	\end{align*} Similarly, one can prove that $(b+c)a = ba + ca$. Thus $H$ is a nonassociative hemiring and the proof is complete. 
\end{proof}

\begin{example}\label{crossproductBooleansemiring}
Note that the multiplication in (\ref{semialgebramultiplication}) is not necessarily associative. For example, let $\mathbb{B} = \{0,1\}$ be the Boolean semifield. We define addition on $\mathbb{B}^3$ componentwisely and the cross product ``$\times$'' on $\mathbb{B}^3$ as follows: \[(u_1,u_2,u_3) \times (v_1,v_2,v_3) = (u_2v_3 + u_3 v_2, u_3 v_1 + u_1 v_3, u_1 v_2 + u_2 v_1).\] It is, then, easy to see that $(\mathbb{B}^3,+,\times)$ is a nonassociative hemiring. The cross product ``$\times$'' on $\mathbb{B}^3$ is not associative because \[(\mathbf{i} \times \mathbf{j}) \times \mathbf{j} = \mathbf{k} \times \mathbf{j} =  \mathbf{i} \text{~and~}\] \[\mathbf{i} \times (\mathbf{j} \times \mathbf{j}) = \mathbf{i} \times \mathbf{0} = \mathbf{0},\] where $\mathbf{i}$, $\mathbf{j}$, and $\mathbf{k}$ are the standard basis for $\mathbb{B}^3$.
\end{example}

Similar to ring and semiring theory, one can prove the following:

\begin{proposition}
	Let $H$ be an NA-hemiring and $X$ an indeterminate over $H$. Then, the set of all polynomials $H[X]$ and formal power series $H[[X]]$ equipped with their standard additions and multiplications are NA-hemirings.
\end{proposition}

\begin{definition}
	Let $(R,+,\cdot)$ be a ringoid.
	
	\begin{enumerate}
		\item $R$ is multiplicatively unital, if there is an element $1 \in R$ with $1r = r1 = r$, for all $r \in R$.
		
		\item $R$ is multiplicatively idempotent if $rr = r$, for all $r \in R$.
		
		\item A multiplicatively unital ringoid $R$ with zero is complemented if for any $r \in R$, there is a unique $r' \in R$ with \[r\cdot r' = r' \cdot r = 0 \text{~and~} r + r' = r' + r = 1.\]
		
		\item A multiplicatively unital ringoid $(S,+,\cdot)$ with zero is an NA-semiring (i.e., a nonassociative semiring \cite{Dobbertin1979}) if $0 \neq 1$ and $(S,+,0)$ is a commutative monoid. One may consider the family of NA-semirings a generalization of NA-rings discussed extensively in Schafer's classical book \cite{Schafer1966}. 
		
		\item If the multiplication of an NA-semiring $S$ is associative by chance, then $S$ is simply called a semiring \cite{Golan1999(b)}.
	\end{enumerate}
\end{definition}

Now, we proceed to give a particular family of NA-semirings called Newman algebras \cite[13*]{Birkhoff1967}. Let us recall that an algebraic system $(N,+,\cdot,0,1)$ is a Newman algebra \cite{Newman1941,Newman1942} if the following conditions hold:

\begin{enumerate}
	\item $(N,+,0)$ is a unital magma,
	
	\item The element $1$ is a right identity element for the magma $(N,\cdot)$,
	
	\item Multiplication distributes on addition from both sides,
	
	\item For any $a \in N$, there is an element $a' \in N$ with $aa' = 0$ and $a+a' = 1$.
\end{enumerate}

\begin{theorem}\label{newmanalgebrasnasemirings}
	Any Newman algebra $(S,+,\cdot,0,1)$ with $0 \neq 1$ is a complemented  multiplicatively idempotent NA-semiring.
\end{theorem}

\begin{proof}
	By T1 on p. 49 in \cite{Birkhoff1967}, $S$ is multiplicatively idempotent. By T2 and N$4'$ and the discussion between the mentioned results, $S$ is complemented. By T3, the element $0$ is an absorbing element of the magma $(S,\cdot)$. By N$2'$, the element $1$ is an identity element for $(S,\cdot)$, and so, $(S,\cdot,1)$ is a unital magma. By Exercise 8 on p. 52 in \cite{Birkhoff1967} (for a complete proof, see P17 and P18 in \cite{Newman1941}), $(S,+,0)$ is a commutative monoid and the proof is complete.
\end{proof}

\section{Some remarks on ideals of ringoids}\label{sec:idealsofringoids}

\begin{proposition}
	Let $\{I_\alpha\}$ be a family of left (right) ideals of a ringoid $R$ such that $\bigcap_{\alpha} I_\alpha$ is nonempty. Then, $\bigcap_{\alpha} I_\alpha$ is a left (right) ideal of $R$. The smallest left (right) ideal of a ringoid $R$ containing a nonempty subset $X$ of $R$, denoted by $(X)_l$ ($(X)_r$), is \[(X)_l = \bigcap_{X \subseteq I \in \Id_l(S)} I \quad \left((X)_r = \bigcap_{X \subseteq I \in \Id_r(S)} I\right).\]
\end{proposition}

\begin{proof}
It is straightforward to see that $\bigcap_{\alpha} I_\alpha$ is a submagma of $(R,+)$. Now, let $r \in R$ and $x \in \bigcap_{\alpha} I_\alpha$. Since $I_\alpha$ is a left (right) ideal of $R$ for each $\alpha$, the element $rx$ ($xr$) is in $I_\alpha$, for each $\alpha$. Therefore, $rx~(xr) \in \bigcap_{\alpha} I_\alpha$ showing that $\bigcap_{\alpha} I_\alpha$ is a left (right) ideal of $R$. Now, it is evident that $(X)_l = \bigcap_{X \subseteq I \in \Id_l(S)} I$ ($(X)_r = \bigcap_{X \subseteq I \in \Id_r(S)} I$) is the smallest left (right) ideal of $R$ containing $X$. This completes the proof.
\end{proof}

\begin{corollary}
	If $I_\alpha$ is an ideal of a ringoid $R$ for each $\alpha$, then their intersection $\bigcap_{\alpha} I_\alpha$ is also an ideal of $R$ if it is nonempty. In particular, if $X$ is a nonempty subset of $R$, the smallest ideal of $R$ containing $X$, denoted by $(X)$, exists and is \[(X) = \bigcap_{X \subseteq I \in \Id(S)} I.\]
\end{corollary}

\begin{remark}
If in a ringoid $R$, there is an element $0$ being an absorbing element of the multiplicative magma of $R$, i.e., $r0 = 0r = 0$, for all $r \in R$, then each ideal $I$ of $R$ contains the element $0$, and so, any intersection of ideals of such ringoids is automatically nonempty. 
\end{remark}

\begin{definition}\label{pir}
The smallest ideal of a ringoid $R$ containing an element $x \in R$, denoted by $(x)$, is called the principal ideal of $R$ generated by $x$. A ringoid $R$ is called to be a principal ideal ringoid if each ideal of $R$ is principal. 
\end{definition}

Similar to ring and semiring theory, the addition of two nonempty subsets $I$ and $J$ of a ringoid $R$, denoted by $I+J$, is defined as follows: \[I+J = \{a+b : a\in I, b\in J\}.\]

\begin{definition}\label{additivelymedialringoid}
We define a ringoid $(R,+,\cdot)$ to be additively medial if the magma $(R,+)$ is medial.	
\end{definition}

\begin{proposition}
Let $I$ and $J$ be left (right) ideals of an additively medial ringoid $R$. Then, $I+J$ is also a left (right) ideal of $R$.
\end{proposition}

\begin{proof}
Let $a+b$ and $c+d$ be elements of $I+J$. Since $R$ is additively medial, we have \[(a+b)+(c+d) = (a+c)+(b+d) \in I+J\] showing that $I+J$ is a submagma of $(R,+)$. Now, let $a+b$ be an element of $I+J$. Since $I$ and $J$ are left (right) ideals of $R$, we have: \[r(a+b) = ra + rb \in I+J \qquad ((a+b)r = ar + br \in I+J).\] Thus $I+J$ is a left (right) ideal of $R$ and the proof is complete.
\end{proof}

\begin{corollary}
Let $R$ be a ringoid such that its additive magma is a commutative semigroup. Then, addition of two left (right) ideals of $R$ is a left (right) ideal of $R$.
\end{corollary}

\begin{proof}
If the additive magma of $R$ is a commutative semigroup, then the ringoid $R$ is additively medial. 
\end{proof}

The multiplication of two nonempty subsets $I$ and $J$ of a ringoid $R$, denoted by $IJ$, is defined as follows: \[IJ = \{ab : a \in I, b \in J\}.\]

\begin{proposition}\label{multiplicationidealsubsetintersection}
Let $I$ be a right and $J$ be a left ideal of a ringoid $R$. Then, $IJ \subseteq I \cap J$.
\end{proposition}

\begin{proof}
Since $I$ is a right ideal of $R$, we have \[IJ \subseteq IR \subseteq I.\] Similarly, since $J$ is a left ideal of $R$, we have $IJ \subseteq RJ \subseteq J$. This completes the proof.
\end{proof}

Eilenberg and Sch\"{u}tzenberger  defined and investigated subtractive submonoids of commutative monoids in \cite{EilenbergSchutzenberger1969}. We add that subtractive subsemigroups have applications in automata theory \cite{Sakarovitch2009} and subtractive ideals of semirings are defined similarly \cite[p. 66]{Golan1999(b)}. In the following, we define subtractive ideals of ringoids:

\begin{definition}\label{subtractiveidealsdef}
	Let $(R,+,\cdot)$ be a ringoid. A left (right) ideal $I$ of $R$ is subtractive if $x+y \in I$ and $x\in I$ imply that $y\in I$, and $x+y \in I$ and $y\in I$ imply that $x\in I$, for all $x,y\in R$.
\end{definition}

\begin{proposition}\label{intersectionsubtractiveideals}
	If $I_\alpha$ is a left (right) subtractive ideal of a ringoid $R$ for each $\alpha \in A$ and $\bigcap_{\alpha \in A} I_\alpha \neq \emptyset$, then $\bigcap_{\alpha \in A} I_\alpha$ is also a subtractive left (right) ideal of $R$.
\end{proposition}

\begin{proof}
	Straightforward.
\end{proof}

\begin{definition}
We define a ringoid $R$ with zero $0$ to be entire if $ab = 0$ implies either $a = 0$ or $b = 0$, for all $a$ and $b$ in $R$. We call $R$ to be zerosumfree if $a + b = 0$ implies $a = b = 0$, for all $a$ and $b$ in $R$. We say $R$ is left (right) austere if the only subtractive left (right) ideals of $R$ are $\{0\}$ and $R$. 
\end{definition}

The following is a nonassociative version of LaGrassa's general example for austere semirings, i.e., semirings without nontrivial subtractive ideals (cf. Example 6.24 in \cite{Golan1999(b)} and Example 1.1 in \cite{LaGrassa1995}):  

\begin{proposition}\label{austereringoid} 
Let $(M,\cdot,1)$ be a unital magma with an absorbing element $0 \neq 1$. Define addition on $M$ by $a+b = 0$, for all $a$ and $b$ in $M$. Suppose that $z \notin M$ and set $S = M \cup \{z\}$. Extend addition and multiplication of $M$ to $S$ as follows:

\begin{itemize}
	\item $s + z = z + s = s$, for all $s \in S$.
	\item $s z = z s= z$, for all $s \in S$.
\end{itemize}

Then, the following statements hold:

\begin{enumerate}
	\item $(S,+,\cdot,z,1)$ is a zerosumfree entire nonassociative semiring.
	\item $S$ is left austere, i.e., any proper left ideal $I$ of $S$ which is nonzero in $S$, i.e., $\{z\} \neq I \neq S$, is not subtractive.
\end{enumerate}
\end{proposition}

\begin{proof}
The first statement is straightforward. Let $I$ be a proper left ideal of $S$ with $\{z\} \neq I$. Let $s \in S$ but $s \notin I$. Also, let $a \in I$ with $a \neq z$. Since $a \in I$ and $a \neq z$, we have $0a\in I$ with $0a = 0$. It follows that $0 \in I$. On the other hand, $a+s = 0 \in I$, while $s \notin I$. So, $I$ is not subtractive. This completes the proof.
\end{proof}

If the additive operation $+$ on a ringoid $(R,+,\cdot)$ is associative, then the term \[x_1 + x_2 + \dots + x_n\] can be unambiguously computed (\cite[Theorem 1]{Bourbaki1998}). However, if $+$ is not associative, then the value of the term \[x_1 + x_2 + \dots + x_n\] is ambiguous and in order to clarify the value of the term, we need to specify the order of operations using parentheses. Based on this discussion, let \[A(x_1,x_2,\dots,x_n)\] represent the set of all possible ways to parenthesize the term $x_1 + x_2 + \dots + x_n$ to be able to compute it unambiguously (cf. Corollary 6.2.3(iii) in \cite{Stanley2001ii}). 

\begin{proposition}\label{generalsubtractivity}
	Let $I$ be a subtractive left (right) ideal of a ringoid $(R,+,\cdot)$. Let $\{x_i\}_{i=1}^{n}$ be a finite family of elements in $R$ and $x$ an arbitrary element of the set \[A(x_1,x_2,\dots,x_n).\] If $x \in I$ and $x_i \in I$ for each $i \neq m$, then $x_m \in I$. 
\end{proposition}

\begin{proof}
	A straightforward strong induction on $n$ finishes the job.
\end{proof}

The following is an easy but a useful ``covering condition'' for subtractive ideals of ringoids:

\begin{theorem}\label{Avoidanceofsubtractiveideals}
	Let $A$, $B$, and $C$ be left (right) ideals of a ringoid $(R,+,\cdot)$ such that $C \subseteq A \cup B$ and $A$ and $B$ are subtractive. Then, either $C \subseteq A$ or $C \subseteq B$.	
\end{theorem}

\begin{proof}
	On the contrary, assume that $C$ is not contained in any of $A$ and $B$. Then, we can find an $x$ in $C$ such that $x$ is not in $A$ but in $B$, and also, we can find a $y$ in $C$ such that $y$ is not in $B$ but in $A$. Observe that \[c = x + y \in C \subseteq A \cup B.\] Now, if $c \in A$, then from $y \in A$, we obtain that $x \in A$ (because $A$ is subtractive), a contradiction. Also, if $c \in B$, then from $x \in B$, we obtain that $y \in B$ (because $B$ is subtractive), again a contradiction. This completes the proof.
\end{proof}

\begin{theorem}\label{annihilatoridealssubtractive}
Let $X$ be a nonempty subset of a ringoid $(R,+,\cdot)$ with zero. Also, let $(R,\cdot)$ be a semigroup. Then, \[\Ann_l(X) = \{r \in R : rX = \{0\}\} \quad (\Ann_r(X) = \{s \in R : Xs = \{0\}\})\] is a subtractive left (right) ideal of $R$.
\end{theorem}

\begin{proof}
We prove the statement for $L = \Ann_l(X)$. The statement for $\Ann_r(X)$ is proved similarly. If $a,b \in L$, then $aX = bX = \{0\}$. Since $0$ is a neutral element of the magma $(R,+)$, we have \[(a+b)X \subseteq aX + bX = \{0\} + \{0\} = \{0\}.\] This shows that $a+b \in L$. Now, let $r \in R$. Since multiplication of $R$ is associative and $0$ annihilates all elements of $R$, we have \[(ra)X = r(aX) = r\{0\} = \{0\}.\] So, $ra \in L$. It follows that $\Ann_l(X)$ is a left ideal of $R$. Now, suppose that $a+b \in L$. If $a \in L$, then, \[(a+b)X = \{0\} \text{~and~} aX = \{0\}.\] Using distributivity, we can deduce that $bX = \{0\}$ showing that $b \in L$. Similarly, if $b \in L$, then $a \in L$. Thus $\Ann_l(X)$ is subtractive and the proof is complete.
\end{proof}

\begin{definition}
Let $(R,+,\cdot)$ be a ringoid with zero such that $(R,\cdot)$ is a semigroup. A left (right) ideal $I$ of $R$ is a left annihilator ideal of $R$ if there is a nonempty subset $X$ of $R$ with \[ I = \Ann_l(X) \quad (I = \Ann_r(X)).\]
\end{definition}

\begin{proposition}\label{intersectionannihilatorideals}
Let $(R,+,\cdot)$ be a ringoid with zero such that $(R,\cdot)$ is a semigroup. An arbitrary intersection of left (right) annihilator ideals of $R$ is a left (right) annihilator ideal of $R$. 
\end{proposition}

\begin{proof}
Let $\{X_\alpha\}$ be a family of nonempty subsets of $R$. It is easy to verify that \[\bigcap_\alpha \Ann_l(X_\alpha) = \Ann_l (\cup_\alpha X_\alpha) \text{~and~} \bigcap_\alpha \Ann_r(X_\alpha) = \Ann_r(\cup_\alpha X_\alpha).\] This completes the proof.
\end{proof}

Inspired by the definition of prime ideals for nonassociative rings (see Definition 1 in \cite{Behrens1956}), we give the following definition:

\begin{definition}\label{primeidealringoiddef}
A proper ideal $P$ of a ringoid $R$ is called to be prime if $(a)(b) \subseteq P$ implies either $a \in P$ or $b \in P$, for all $a$ and $b$ in $R$, where by $(x)$, we mean the principal ideal of $R$ generated by $x \in R$.
\end{definition}

\begin{lemma}\label{primeidealringoid}
Let $P$ be a prime ideal of a ringoid $R$. If $a, b \notin P$, then there are $a'$ and $b'$ in the principal ideals $(a)$ and $(b)$, respectively, with $a'b' \notin (a)(b)$.
\end{lemma}

\begin{proof}
Otherwise $(a)(b) \subseteq P$, and so by primeness of $P$, either $a\in P$ or $b \in P$.
\end{proof}

\begin{theorem}\label{primeidealringoidthm}
Let	$P$ be a proper ideal of a ringoid $R$. Then, the following statements are equivalent:

\begin{enumerate}
	\item $P$ is a prime ideal of $R$;
	
	\item $IJ \subseteq P$ implies either $I \subseteq P$ or $J \subseteq P$, for all ideals $I$ and $J$ of $R$;
	
	\item $(IJ) \subseteq P$ implies either $I \subseteq P$ or $J \subseteq P$, for all ideals $I$ and $J$ of $R$, where by $(IJ)$, we mean the smallest ideal of $R$ containing $IJ$.
\end{enumerate}

\end{theorem}

\begin{proof}
 $(3) \Rightarrow (2)$: Let $IJ \subseteq P$. Note that $(IJ)$ is the smallest ideal containing $IJ$. Therefore, $(IJ) \subseteq P$. Now, $P$ contains one of the ideals $I$ and $J$.
 
$(2) \Rightarrow (3)$: Evident.

$(2) \Rightarrow (1)$: Evident. 

$(1) \Rightarrow (2)$: Let $I$ and $J$ be arbitrary ideals of $R$ with $I \nsubseteq P$ and $J \nsubseteq P$. This means that there are $a \in I$ and $b \in J$ with $a,b \notin P$. By Lemma \ref{primeidealringoid}, there are $a'$ and $b'$ in the principal ideals $(a)$ and $(b)$, respectively, with $a' b' \notin P$. Clearly, \[a' b' \in (a)(b) \subseteq IJ,\] and so, $IJ \nsubseteq P$. This completes the proof.
\end{proof}

\begin{theorem}\label{Behrenlemma}
Let $\{P_i\}_{i=1}^{n}$ be a finite family of prime ideals of a ringoid $R$. Let $I$ be an ideal of $R$ such that for each $i \in \mathbb{N}_n$, there is an element $a_i \in I$ with $a_i \in P_i$ but $a_i \notin P_k$ for each $k \neq i$. Then, for each $l \in \mathbb{N}_n$, there is an element $b_l$ belonging to $I$ and each $P_i$ except $P_l$.
\end{theorem}

\begin{proof}
Observe that since $P_n$ is a prime ideal of $R$, by Lemma \ref{primeidealringoid}, there are two elements $a'_1$ and $a'_2$ in the principal ideals $(a_1)$ and $(a_2)$, respectively, with $a'_1 \cdot a'_2 \notin P_n$. Using this, we can find $a''_2$ and $a'_3$ in the principal ideals $(a'_1 \cdot a'_2)$ and $(a_3)$, respectively, with $a''_2 \cdot a'_3 \notin P_n$. Continuing the process, we can find $a''_{n-2}$ and $a'_{n-1}$ in the principal ideals $(a'_{n-3} \cdot a'_{n-2})$ and $(a_{n-1})$, respectively, with $b_n = a''_{n-2} \cdot a'_{n-1} \notin P_n$. It is evident that the process of constructing the element $b_n = a''_{n-2} \cdot a'_{n-1}$ guarantees to be an element of $I$ and each $P_i$ with $i < n$. By relabeling, the same argument works to find an element $b_l$ belonging to $I$ and all $P_i$s except $P_l$, for each $l < n$. This completes the proof. 
\end{proof}

\begin{remark}
The proof of Theorem \ref{Behrenlemma} is based on Behrens' technique used in ``Satz 3'' in \cite{Behrens1956}.
\end{remark}

Let $X$ be a subset of a semiring $S$. It is straightforward to see that the smallest ideal of $S$ containing $X$ is \[(X) = \left\{\sum_{i=1}^{n} s_i x_i t_i : s_i,t_i \in S, x_i \in X, n \in \mathbb{N}\right\},\] if $X$ is nonempty and the zero ideal $\{0\}$ if $X$ is empty. 

\begin{theorem}\label{primenesscriterionsemiringthm}
	For a proper ideal $P$ of a semiring $S$, the following statements are equivalent:
	
	\begin{enumerate}
		\item $P$ is a prime ideal.
		\item $IJ \subseteq P$ implies either $I \subseteq P$ or $J \subseteq P$ for any left (right) ideals $I$ and $J$ of $S$.
		\item $xSy \subseteq P$ implies $x\in P$ or $y\in P$ for all $x,y \in S$.
	\end{enumerate}
\end{theorem}

\begin{proof}
	$(1) \Rightarrow (2)$: Let $I$ and $J$ be left ideals of $S$. Since multiplication of $S$ is associative, for the ideals $(IS)$ and $(JS)$ of $S$, we have \[(IS)(JS) \subseteq (IJ) \subseteq P.\] Therefore by Theorem \ref{primeidealringoidthm}, either $IS \subseteq P$ or $JS \subseteq P$. Now, since $S$ is a semiring and has a multiplicative identity, we have $I \subseteq IS$ and $J \subseteq JS$. Thus either $I \subseteq P$ or $J \subseteq P$. In a similar way, one may prove the implication for right ideals.
	
	$(2) \Rightarrow (3)$: Let $xSy \subseteq P$. This implies that $(Sx)(Sy) \subseteq SP \subseteq P$. Since $Sx$ and $Sy$ are left ideals of $S$, we have either $Sx \subseteq P$ or $Sy \subseteq P$. Therefore, either $x\in P$ or $y \in P$.
	
	$(3) \Rightarrow (1)$: If the ideals $I$ and $J$ are not subsets of $P$, then there are elements $x$ and $y$ in $S$ with $x \in I \setminus P$ and $ y \in J \setminus P$. Therefore, by assumption, $xSy$ is not a subset of $P$. Thus $IJ$ is not a subset of $P$ and the proof is complete.    
\end{proof}

Let $X$ be a nonempty subset of an $S$-semimodule $M$. By definition, \[\Ann(X) = \{s\in S: sX = \{0\}\}.\]

\begin{proposition}\label{annihilatorsubtractiveideal}
Let $M$ be an $S$-semimodule and $X$ its $S$-subsemimodule. Then, $\Ann(X)$ is a subtractive ideal of $S$.
\end{proposition}

\begin{proof}
It is straightforward to see that $\Ann(X)$ is a left ideal of $S$. Now, let $a \in \Ann(X)$ and $s \in S$. Note that $sX \subseteq X$ because $S$ is a subsemimodule. However, \[(as)X = a(sX) \subseteq aX = \{0\}\] because the multiplication of $S$ is associative. This implies that the element $as$ is in $\Ann(X)$. So, $\Ann(X)$ is also a right ideal of $S$. Similar to Theorem \ref{annihilatoridealssubtractive}, it is easy to see that $\Ann(X)$ is subtractive. This completes the proof.
\end{proof}

\begin{remark}
By definition, an ideal $I$ of $S$ is an $M$-annihilator ideal if $I = \Ann(X)$, for some nonempty $X \subseteq M$. The following is a semiring version of Proposition 3.2.23 in \cite{Rowen1991}: 
\end{remark}

\begin{theorem}\label{maximalannihilatoridealssubtractiveprime}
Let $S$ be a semiring and $M$ a nonzero $S$-semimodule. Let $P$ be a maximal element of \[\Gamma = \{\Ann(N) : \{0\} < N < M\}.\] Then, $P$ is a subtractive prime ideal of $S$, where by ``$K < L$'', we mean ``$K$ is a proper subsemimodule of $L$''.
\end{theorem}

\begin{proof}
Let $P = \Ann(N)$ be a maximal element of $\Gamma$. Since $M$ is nonzero, $1 \notin P$. So by Proposition \ref{annihilatorsubtractiveideal}, $P$ is a proper subtractive ideal of $S$. Now, assume that $IJ \subseteq P$, for some ideals $I$ and $J$ of $S$. This means that $(IJ)N = \{0\}$. If $JN = \{0\}$, then $J \subseteq P$. Otherwise, $JN \neq \{0\}$. Since $JN$ is an $S$-subsemimodule of $N$, $\Ann(N) \subseteq \Ann(JN)$. Since $JN$ is nonzero, $\Ann(JN) \in \Gamma$. So, by maximality of $\Ann(N)$, we must have $\Ann(N) = \Ann(JN)$. However, the multiplication of $S$ is associative, and so, $I(JN) = \{0\}$ which means that \[I \subseteq \Ann(JN) = \Ann(N) = P.\] Consequently, $P$ is prime. This completes the proof.
\end{proof}

We say $T$ is multiplicatively closed in a semiring $S$ if $T$ is a submonoid of $(S,\cdot,1)$.

\begin{lemma}[Krull's Separation Lemma]\label{maximaltoexclusionprime}
	Let $T$ be a multiplicatively closed set in a commutative semiring $S$. Also, let $I$ be an ideal of $S$ with $I \cap T = \emptyset$. Then, the ideal $P$ maximal with respect to exclusion of $T$ exists and is prime. 
\end{lemma}

\begin{proof}
	Let $T$ be a multiplicatively closed set and assume that $I \cap T = \emptyset$. Put \[\Gamma = \{J \in \Id(S) : J \cap T = \emptyset\}.\] Since $\Gamma$ is nonempty, by Zorn's lemma $\Gamma$ possesses an ideal of $S$ maximal with respect to disjointness from $T$. This proves the existence. 	
	
	Now, let $ab \in P$ with $a,b \notin P$. So, the ideals $P+Sa$ and $P+Sb$ strictly contain $P$. Therefore, they need to intersect $T$ which means that there are elements $t_1,t_2\in T$ of the form $t_1 = s_1 + x_1 a$ and $t_2 = s_2+x_2 b$ where $s_1,s_2 \in P$ and $x_1,x_2 \in S$. Observe that \[t_1 t_2 = s_1 s_2 + s_1 x_2 b + x_1 a s_2 + x_1x_2 ab \in P,\] contradicting that $P$ does not intersect $T$. This shows that $P$ is prime \cite[Corollary 7.6]{Golan1999(b)} and the proof is complete.
\end{proof}

The following is a semiring version of the definition of an $S$-semiprime ideal of a commutative ring $R$ given in Definition 2.1 in \cite{PathakGoswami2023}:

\begin{definition}\label{mcssemiprimedef}
Let $S$ be a commutative semiring, $T$ a multiplicatively closed subset of $S$, and $P$ an ideal of $S$ with $P \cap T = \emptyset$. We say $P$ is a $T$-semiprime ideal of $S$ if there is a $t \in T$ such that $s^2 \in P$ implies $ts \in P$, for all $s \in S$.
\end{definition}

\begin{example}
Let $T$ be a multiplicatively closed subset of a commutative semiring $S$. It is clear that all semiprime ideals of $S$ are $T$-semiprime.
\end{example}

Note that if $I$ is an ideal of a commutative semiring $S$ and $t$ is an element of $S$, then \[[I:t] = \{s\in S: ts \in I\}\] is also an ideal of $S$. The following is a semiring version of Proposition 2.4 in \cite{PathakGoswami2023}: 

\begin{theorem}\label{residualquotientsemiprime} 
Let $S$ be a commutative semiring, $T$ a multiplicatively closed subset of $S$, and $P$ a $2$-absorbing ideal of $S$ disjoint from $T$. Then, $P$ is a $T$-semiprime ideal of $S$ if and only if $[P:t]$ is a semiprime ideal of $S$ for some $t \in T$.
\end{theorem}

\begin{proof}
	For the direct implication, suppose that $P$ is $T$-semiprime. By definition, there is a $t \in T$ such that $s^2 \in P$ implies $ts \in P$, for all $s \in S$. In order to prove that $[P:t]$ is semiprime, take $x \in S$ such that $x^2 \in [P:t]$. It follows that $tx^2 \in P$. Since $P$ is an ideal of $S$ and $S$ is commutative, $(tx)^2 = t(tx^2) \in P$. Now since $P$ is $T$-semiprime, we have $t(tx) \in P$. On the other hand, $P$ is $2$-absorbing. Therefore, either $tt \in P$ or $tx \in P$. Since $t^2 \in T$ and $P \cap T = \emptyset$, we have $t^2 \notin P$. So, $tx \in P$, i.e., $x \in [P:t]$. Conversely, assume that $[P:t]$ is semiprime for some $t \in T$ and take $x \in S$ such that $x^2 \in P$. Consequently, $tx^2 \in P$. So, $x^2 \in [P:t]$, and so, $x \in [P:t]$. This means that $tx \in P$. Thus $P$ is $T$-semiprime and the proof is complete.
\end{proof}

\section{Prime avoidance for ringoids}\label{sec:pat}

\begin{theorem}[Prime avoidance for ringoids]\label{PALforringoids}
	Let $P_1, \dots, P_n$ be subtractive prime ideals of a ringoid $(R,+,\cdot)$. Let $I$ be an ideal of $R$ with $I \nsubseteq P_i$, for each $i \in \mathbb{N}_n$. Then, $a \in I \setminus \bigcup_{i=1}^{n} P_i$, for some $a \in R$.
\end{theorem}

\begin{proof}
	The proof is by strong induction on $n$. The case $n = 1$ is obvious. The case $n = 2$ is an obvious corollary to Lemma \ref{Avoidanceofsubtractiveideals}. Let the statement hold for all positive integer numbers less than $n \geq 3$. By induction's hypothesis, for each $i \in \mathbb{N}_n$, there is an element $a_i$ in $I$ such that $a_i$ is not in $P_k$ for each $k \neq i$. If for at least one $i$, $a_i \notin P_i$, then we can take $a = a_i$ and we are done. Otherwise, we have $a_i \in P_i$, for each $i \in \mathbb{N}_n$. By Theorem \ref{Behrenlemma}, for each $l \in \mathbb{N}_n$, there is an element $b_l$ belonging to $I$ and each $P_i$ except $P_l$. Let $a$ be an element of \[A(b_1, b_2, \dots, b_n)\] obtained from the combination of $b_i$s in the magma $(R,+)$. Since each ideal $P_i$ is subtractive in $R$, by Proposition \ref{generalsubtractivity}, none of the $n$ ideals $\{P_i\}_{i=1}^{n}$ contains the element $a$. However, $a \in I$. This completes the proof.
\end{proof}

The following is an NA-semiring version of Behrens' ``Satz 3'' in \cite{Behrens1956}:

\begin{corollary}
	If an ideal $I$ of an NA-semiring $S$ is not a subset of a subtractive prime ideal $P_i$ of $S$, for each $i \in \mathbb{N}_n$, then $I$ cannot be a subset of the unions of $P_i$s. 
\end{corollary}

The associativity of multiplication in semirings enables us to extend the prime avoidance in semiring theory in another direction:

\begin{theorem}[Prime avoidance for semirings] \label{PALforsemirings}
	Let $S$ be a semiring, $\{P_i\}_{i=1}^n$ a finite family of subtractive ideals of $S$. If for each $i > 2$, $P_i$ is prime and $I$ is an ideal of $S$ with $I \subseteq \bigcup_{i=1}^n P_i$, then $I \subseteq P_i$, for some $1 \leq i \leq n$.
\end{theorem}

\begin{proof}
	The proof is by strong induction on $n$. The case $n =1$ is obvious and the case $n =2$ is nothing but Lemma \ref{Avoidanceofsubtractiveideals}. Now, assume that $n \geq 3$ and $I \subseteq  \bigcup_{i=1}^n P_i$ but $x_i \in I \setminus  \bigcup_{j\neq i} P_j$, for each $i$. Since $x_i \in I$, we see that $x_i \in P_i$, for each $i$. In view of Theorem \ref{primenesscriterionsemiringthm}, primeness of $P_n$ implies that \[x_1Sx_2S \cdots x_{n-2}Sx_{n-1} \nsubseteq P_n.\] So, there is an $x \in x_1Sx_2S \cdots x_{n-2}Sx_{n-1} \setminus P_n$. It follows that $ x \in \bigcap_{j=1}^{n-1} P_j$. Since $P_j$ is subtractive, $x + x_n \notin P_j$, for each $j$ which is obviously a contradiction because $x + x_n$ is an element of $I$. Thus there is an $i$ with \[I \subseteq  \bigcup_{j \neq i} P_j\] and the result follows by induction and the proof is complete.  
\end{proof}

\begin{example}\label{nosubtractivenoPAL}
Some classical results related to ``prime avoidance'' in ring theory can be generalized to subtractive semirings as we see in the current paper. This class of semirings is quite large including rings and bounded distributive lattices. However, inspired by Remark 2.4 in \cite{FontanaZafrullah2015} and with the help of Proposition \ref{austereringoid}, we construct a general example to show that Theorem \ref{PALforsemirings} may not hold if the prime ideals $\{P_i\}_{i=1}^{n}$ are not subtractive. Let us recall that a nonempty subset $I$ of a monoid $(M,\cdot)$ is a monoidal ideal of $M$ if \[ma, am \in I, \qquad\forall~m \in M, ~a \in I.\] It is straightforward to see that an arbitrary union of monoidal ideals of a monoid $M$ is a monoidal ideal of $M$. Now, let $R$ be a semi-local commutative ring with $1 \neq 0$ such that it has exactly $n$ maximal ideals $\{\mathfrak{m}_i\}_{i=1}^{n}$. It is clear that $\mathfrak{a} = \bigcup_{i=1}^{n} \mathfrak{m}_i$ is a monoidal ideal of the multiplicative monoid $(R,\cdot)$. Since $(R,\cdot)$ is a monoid with $0 \neq 1$, as Proposition \ref{austereringoid}, we construct the semiring $S = R \cup \{z\}$. Set $I = \mathfrak{a} \cup \{z\}$ and $M_i = \mathfrak{m}_i \cup \{z\}$, for each $i \in \mathbb{N}_n$. Now, it is straightforward to see that $I$ is an ideal of $S$ and $M_i$s are prime ideals of $S$ with $I \subseteq \bigcup_{i=1}^{n} M_i$ while $I \nsubseteq M_i$, for each $i \in \mathbb{N}_n$.
\end{example}

\begin{question}
Let $I$ and $\{P_i\}_{i=1}^{n}$ be ideals of a semiring $S$ such that each $P_i$ is prime. Example \ref{nosubtractivenoPAL} shows that the ``finite prime avoidance'' (i.e., $I \subseteq \bigcup_{i=1}^{n} P_i$ implies $I \subseteq P_i$, for some $i$) does not hold if some of the $P_i$s are not subtractive. So, the full characterization of semirings in which this property holds is still open for the author.   
\end{question}

\begin{remark}
Our proof for Theorem \ref{PALforsemirings} is inspired by the proof of Proposition 2.12.7 in \cite{Rowen1991}. We have applied prime avoidance for commutative rings and semirings in some of our previous works including  \cite{Nasehpour2025,Nasehpour2016,Nasehpour2021,Nasehpour2010,Nasehpour2011,NasehpourPayrovi2010}. Hetzel and Lewis Lufi (cf. Lemma 3.11 in \cite{HetzelLufi2009}) and Ye\c{s}ilot (cf. Theorem 2.6 in \cite{Yesilot2010}) proved a version of Theorem \ref{PALforsemirings} for commutative semirings. The following is some kind of semiring version of Proposition 3.2.31 in \cite{Rowen1991}:
\end{remark}

\begin{theorem}\label{containedannihilatorcontainedprime}
Let $S$ be a semiring and $M$ an $S$-semimodule. Also, let $S$ satisfy a.c.c. on $M$-annihilator ideals of $S$ (for example, let $S$ be a Noetherian semiring). If $I$ is an ideal of $S$ and a subset of a finite union of $M$-annihilator ideals of $S$, then $I$ is contained in an $M$-annihilator and prime ideal of $S$.
\end{theorem}

\begin{proof}
Since $S$ satisfies a.c.c. on its annihilator ideals, the maximal elements of such ideals exist, and by Theorem \ref{maximalannihilatoridealssubtractiveprime}, they are prime and subtractive. So, $I$ is a subset of a finite union of prime and subtractive ideals of $S$. Thus by Theorem \ref{PALforsemirings}, $I$ is contained in an annihilator and prime ideal of $S$ and the proof is complete.
\end{proof}

For each ideal $I$ of a semiring $S$, set \[V(I) = \{P \in \Spec(S) : P \supseteq I\} \text{~and~} D(I) = \Spec(S) \setminus V(I).\] Note that if $I = (x)$ is a principal ideal of $S$, we denote $V(I)$ and $D(I)$ by $V(x)$ and $D(x)$, respectively. 

We also add that $\mathcal{C} = \{V(I) : I \in \Id(S)\}$ is the family of closed sets for a topology on $X = \Spec(S)$, called the Zariski topology \cite[p. 89]{Golan1999(b)}. The topological interpretation of the prime avoidance is the following:

\begin{theorem}\label{primeavoidancetopology}
Let $S$ be a subtractive semiring. In the Zariski topology on $X = \Spec(S)$, if a finite number of points are contained in an open subset then they are contained in a smaller principal open subset.
\end{theorem}

\begin{proof}
Let $S$ be a subtractive semiring. In the Zariski topology on $X = \Spec(S)$, consider the finitely many points $\{P_i\}_{i=1}^{n}$ contained in an open subset $D(I) = X \setminus V(I)$, for some ideal $I$ of $S$. This means that $I \nsubseteq P_i$, for each $1 \leq i \leq n$. By Theorem \ref{PALforsemirings}, there is an element $x \in I$ such that $x \notin \bigcup_{i=1}^n P_i$. It follows that \[P_i \in D(x) \subseteq D(I), \qquad~1\leq i \leq n,\] where $D(x)$ is the principal open subset smaller than $D(I)$ and the proof is complete.  
\end{proof}

\begin{theorem}[Davis' prime avoidance] \label{Davisprimeavoidancelemma} Let $x$ be an element of a semiring $S$ and $I$ an ideal of $S$. Also, suppose that $P_i$s are subtractive prime ideals of the semiring $S$. Then, $(x) + I \nsubseteq  \bigcup_{i=1}^{n} P_i$ implies that there is a $y \in I$ with $x+y \notin  \bigcup_{i=1}^{n} P_i$.
\end{theorem}

\begin{proof}
Without loss of generality, we may assume that $P_j \nsubseteq P_i$, for all $j \neq i$. Set \[A = \{i \in \mathbb{N}_n: x \in P_i\}.\] Our claim is that for any index $i \in A$, we have $I \nsubseteq P_i$. Otherwise, $I \subseteq P_i$ and $x \in P_i$ imply that $(x) + I \subseteq P_i$, a contradiction with the assumption $(x) + I \nsubseteq  \bigcup_{i=1}^{n} P_i$. Since $P_i$ is a prime ideal of $S$, from all we said, we deduce that \[I \cdot \prod_{j \notin A} P_j \nsubseteq P_i, \qquad \forall~ i \in A.\] Since $P_i$s are subtractive prime ideals of $S$, using Theorem \ref{PALforsemirings}, we obtain that \[I \cdot \prod_{j \notin A} P_j \nsubseteq  \bigcup_{i\in A} P_i.\] Let $y$ be an element of $I \cdot \prod_{j \notin A} P_j \setminus \bigcup_{i\in A} P_i$. By Proposition \ref{multiplicationidealsubsetintersection}, $y$ is in $I$, and also, in $P_j$ with $j \notin A$, but not in $P_i$ with $i \in A$. Our claim is that $x+y \notin P_i$ for each $i$. On the contrary, assume that there is an $i$ such that $x+y \in P_i$. If $i \in A$, then by the definition of $A$, $x \in P_i$, and because $P_i$ is subtractive, $y \in P_i$, a contradiction. If $i \notin A$, then $y \in P_i$, and so, $x \in P_i$ because $P_i$ is subtractive. It follows that $i \in A$, again a contradiction. Hence, $x+y \notin  \bigcup_{i=1}^{n} P_i$, as required.
\end{proof}

\begin{remark}
Irving Kaplansky attributes a commutative ring version of Theorem \ref{Davisprimeavoidancelemma} to his student Edward Dewey Davis. Davis' prime avoidance has applications in grades of ideals in ring theory (see Theorem 124 and Theorem 125 in \cite{Kaplansky1974} and Exercise 16.8 and Exercise 16.9 in \cite{Matsumura1989}).
\end{remark}

\section{Finite unions of subtractive ideals}\label{sec:finiteunions}

McCoy, in his paper \cite{McCoy1957}, investigated finite unions of ideals and proved that if $I \subseteq \bigcup_{i=1}^{n} A_i$ is efficient for some ideals $I$ and $A_i$s of a commutative ring $R$, then $I^k \subseteq \bigcap_{i=1}^{n} A_i$ for some positive integer $k$. One of the corollaries of this result is the ``prime avoidance'', which has been widely applied in various areas of commutative algebra. The main purpose of this section is to generalize some of his results in the context of semiring theory. The following is a semiring version of McCoy's lemma given on p. 634 in \cite{McCoy1957}:

\begin{lemma}\label{finiteintersectionmccoy}
	Let $n \geq 3$ be a positive integer, and $I$ and $\{A_i\}_{i=1}^n$ be subtractive ideals of a commutative semiring $S$ such that $I \subseteq \bigcup_{i=1}^n A_i$ is efficient. Then, the intersection of any $n-1$ of the ideals $A_i$ coincides with $\bigcap_{i=1}^n A_i$.  	
\end{lemma}

\begin{proof}
	Since $I \subseteq \bigcup_{i=1}^n A_i$, it is clear that $I = \bigcup_{i=1}^{n} (I \cap A_i)$. Therefore in view of Proposition \ref{intersectionsubtractiveideals}, without loss of generality, we may assume that $I = \bigcup_{i=1}^n A_i$. Since $I$ is not contained in the union of any $n-1$ ideals $A_i$, there is an element $a_n$ in $I$ and in $A_n$ such that \[a_n \notin  \bigcup_{i=1}^{n-1} A_i.\] Now, let $x \in  \bigcap_{i=1}^{n-1} A_i$. The element $x+a_n$ which is in $I$ is not in $A_i$, for each $i \leq n-1$ because of subtractivity of $A_i$s. Consequently, $x+a_n$ is an element of $A_n$. Since $A_n$ is subtractive, $x \in A_n$. Thus $\bigcap_{i=1}^{n-1} A_i \subseteq  \bigcap_{i=1}^n A_i$ and the proof is complete.
\end{proof}

The following is a semiring version of McCoy's result for commutative rings given on p. 634 in \cite{McCoy1957}: 

\begin{theorem}\label{finiteintersectionmccoythm}
	Let $n \geq 3$ be a positive integer, and $I$ and $\{A_i\}_{i=1}^n$ be ideals of a subtractive commutative semiring $S$ such that $I \subseteq  \bigcup_{i=1}^n A_i$ is efficient. Then, there is a positive integer $k$ such that \[I^k \subseteq  \bigcap_{i=1}^n A_i.\]
\end{theorem}

\begin{proof}
	Similar to Lemma \ref{finiteintersectionmccoy}, we assume that $I =  \bigcup_{i=1}^n A_i$. First we establish the theorem for $n = 3$. Since \[A_1 \cup A_2 \subseteq A_1 + A_2,\] from $I =  \bigcup_{i=1}^3 A_i$, we obtain that $I \subseteq (A_1 + A_2) \cup A_3$. By Lemma \ref{Avoidanceofsubtractiveideals}, if an ideal is contained in the union of two subtractive ideals, it is contained in one of them. Therefore, $I$ is subset of $A_1 + A_2$ because $I$ is not contained in $A_3$. Similarly, $I$ is contained in $A_2 + A_3$ and $A_3 + A_1$. It follows that \[I^3 \subseteq (A_1 + A_2) (A_2 + A_3) (A_3 + A_1).\] If the right side of the above inclusion is multiplied out, then each term consists of a product containing at least two different ideals $A_i$. On the other hand, a product of ideals is contained in their intersection. Now, in view of Lemma \ref{finiteintersectionmccoy}, we have \[I^3 \subseteq A_1 \cap A_2 \cap A_3.\] Now, let the theorem hold for all $3 \leq k < n$. The proof will be complete, if we can deduce that the theorem holds for $n$. Now, assume that $I \subseteq  \bigcup_{i=1}^n A_i$. As above, this implies that \begin{align}
		I \subseteq (A_1 + A_2) \cup A_3 \cup \cdots \cup A_n. \label{finiteintersectionmccoythm1} \end{align} Obviously, either $I \subseteq A_1 + A_2$ or $I$ is contained in the union of some $m < n$ but not in the union of any $m-1$ of the ideals on the right side of the inclusion (\ref{finiteintersectionmccoythm1}). Since \[I \nsubseteq A_3 \cup \cdots \cup A_n,\] one of the remaining $m$ ideals of the right side of (\ref{finiteintersectionmccoythm1}) must be $A_1 + A_2$. Now, by induction's hypothesis, we have $I^{k_1} \subseteq A_1 + A_2$, for some positive integer $k_1$. A similar argument works for each pair of ideals $A_i$ and $A_j$ with $i < j$. Therefore, there is a positive integer $k$ such that \begin{align}
		I^k \subseteq \prod_{i < j} (A_i + A_j) \label{finiteintersectionmccoythm2}
	\end{align} 
	
	Now, if the right side of the inclusion (\ref{finiteintersectionmccoythm2}) is multiplied out, each term consists of a product containing at least $n-1$ different ideals $A_i$. Hence, by Lemma \ref{finiteintersectionmccoy}, we have $I^k \subseteq  \bigcap_{i=1}^n A_i$, as required.
\end{proof}

Let us recall that an ideal $J$ of a commutative semiring $S$ is radical if $\sqrt{J} = J$, where by the radical of an ideal $J$, we mean \[\sqrt{J} = \{s \in S: \exists~n \in \mathbb{N}~(s^n \in J)\}.\]

\begin{corollary}[McAdam's radical ideal avoidance for semirings]\label{mcadamradicalidealavoidancelemma} Let $S$ be a subtractive commutative semiring and $I$ and $\{J_i\}_{i=1}^{n}$ be ideals of $S$. If at least $n-2$ of the ideals in $\{J_i\}_{i=1}^{n}$ are radical and $I$ is contained in $ \bigcup_{i=1}^{n} J_i$, then $I$ is contained in at least one of the $J_i$s.
\end{corollary}

\begin{proof}
Straightforward.
\end{proof}

Since any semiprime ideal of a commutative semiring is a radical ideal, we have the following:

\begin{corollary}[McCoy's semiprime avoidance for semirings]\label{mccoysemiprimeavoidancelemma} Let $S$ be a subtractive commutative semiring and $I$ and $\{P_i\}_{i=1}^{n}$ be ideals of $S$. If at least $n-2$ of the ideals in $\{P_i\}_{i=1}^{n}$ are semiprime and $I$ is contained in $ \bigcup_{i=1}^{n} P_i$, then $I$ is contained in at least one of the $P_i$s.
\end{corollary}

Using McCoy's semiprime avoidance for semirings, we generalize Proposition 2.13 in \cite{PathakGoswami2023} to subtractive commutative semirings:

\begin{theorem}\label{semiprimeavoidance2absorbing}
Let $T$ be a multiplicatively closed subset of a subtractive commutative semiring $S$. Let $\{P_i\}_{i=1}^{n}$ be $T$-semiprime and $2$-absorbing ideals of $S$ and $I$ an ideal of $S$ such that $I \subseteq \bigcup_{i=1}^{n} P_i$. Then, $tI$ is contained in one the $P_i$s, for some $t \in T$.
\end{theorem}

\begin{proof}
Since each $P_i$ is $T$-semiprime and $2$-absorbing, using Theorem \ref{residualquotientsemiprime}, we see that for each $P_i$ there is a $t_i \in T$ such that $[P_i:t_i]$ is a semiprime ideal of $S$. Since $P_i \subseteq [P_i : t_i]$ for each $i$, it follows that \[I \subseteq \bigcup_{i=1}^{n} P_i \subseteq \bigcup_{i=1}^{n} [P_i: t_i].\] By Corollary \ref{mccoysemiprimeavoidancelemma}, $I \subseteq [P_j,t_j]$, for some $1 \leq j \leq n$. Thus $tI$ is contained in one the $P_i$s, for some $t \in T$ and this finishes the proof.
\end{proof}

\section{Compactly packed semirings}\label{sec:compactlypackedsemirings}

\begin{definition}\label{compactlypackedsemirings}
We define a semiring $S$ to be compactly packed if the following covering condition holds for prime ideals of $S$:

\begin{itemize}
	\item For an arbitrary family of prime ideals $\{P_\alpha\}$ and any ideal $I$ of $S$, the inclusion $I \subseteq  \bigcup_{\alpha} P_\alpha$ implies $I \subseteq P_\alpha$, for some $\alpha$.
\end{itemize}

\end{definition}

\begin{example}\label{examplescompactlypacked} In the following, we give some examples:
	
	\begin{enumerate}
		\item Recall that a commutative semiring $S$ is weak Gaussian if and only if each prime ideal of $S$ is subtractive (see Definition 18 and Proposition 19 in \cite{Nasehpour2016}). Now, if $S$ is a weak Gaussian semiring with finitely many prime ideals, then by Theorem \ref{PALforsemirings}, $S$ is compactly packed.
		
		\item Every principal ideal semiring $S$ (see Definition \ref{pir}) is compactly packed because \[(x) \subseteq  \bigcup_{\alpha} P_\alpha\] implies that $x \in P_\alpha$, and consequently, $(x) \subseteq P_\alpha$ for some $\alpha$. Examples of ``proper'' principal ideal semirings include subtractive $\mathbb{O}_{\infty}$-Euclidean semirings (see Corollary 1.5 in \cite{Nasehpour2019}) and discrete valuation semirings (see Theorem 3.6 in \cite{Nasehpour2018V}).
		
		\item The set of non-negative integer numbers $\mathbb{N}_0$ equipped with usual addition and multiplication of numbers is a commutative semiring and $p \mathbb{N}_0$ is a subtractive prime ideal of $\mathbb{N}_0$ for each prime number $p$ (see Proposition 6 in \cite{Noronha1978}). On the other hand, $\mathfrak{m} = \mathbb{N}_0 \setminus \{1\}$ is a maximal ideal of $\mathbb{N}_0$ with $\mathfrak{m} \subseteq \bigcup_{p \in \mathbb{P}} p\mathbb{N}_0$. However, $\mathfrak{m}$ is not a subset of any of the prime ideals $p \mathbb{N}_0$. This means that $\mathbb{N}_0$ is not compactly packed. 
		
	\end{enumerate}

\end{example}

\begin{theorem}\label{generalizedpal} 
	Let $S$ be a commutative semiring. Then, the following statements are equivalent:
	\begin{enumerate}
		\item The semiring $S$ is compactly packed.
		
		\item For an arbitrary family of prime ideals $\{P_\alpha\}$ and any prime ideal $Q$ of $S$, the inclusion $Q \subseteq  \bigcup_{\alpha} P_\alpha$ implies $Q \subseteq P_\alpha$, for some $\alpha$.
		
		\item Each prime ideal of $S$ is the radical of a principal ideal in $S$.
		
		\item Each semiprime ideal of $S$ is the radical of a principal ideal in $S$.
		
		\item Each radical ideal is the radical of a principal ideal.	
	\end{enumerate}	
\end{theorem}

\begin{proof}
	$(1) \Rightarrow (2)$: Evident.
	
	$(2) \Rightarrow (3)$: On the contrary, let $Q$ be a prime ideal of $S$ such that $Q \neq \sqrt{(s)}$, for all $s \in S$. Observe that \[\sqrt{(s)} =  \bigcap_{s\in P \in \Spec(S)} P, \qquad\forall~s \in S.\] This means that for each $s \in Q$, there is a prime ideal $P_s$ such that $s \in P_s$ and $Q \nsubseteq P_s$. On the other hand, $Q \subseteq  \bigcup_{s\in Q} P_s$.
	
	$(3) \Rightarrow (1)$: First, let $Q = \sqrt{(s)}$ be a prime ideal of $S$ included in $ \bigcup_{\alpha} P_\alpha$. This implies that $s \in P_\alpha$, for some $\alpha$. It is, then, easy to see that \[Q = \sqrt{(s)} \subseteq P_\alpha.\] Now, let $I$ be an arbitrary ideal of $S$ included in $ \bigcup_{\alpha} P_\alpha$. Note that \[T = S \setminus\bigcup_{\alpha} P_\alpha\] is a multiplicatively closed set in $S$ disjoint from $I$. By Lemma \ref{maximaltoexclusionprime}, there is a prime ideal $P$ of $S$ with \[I \subseteq P \subseteq  \bigcup_{\alpha} P_\alpha.\] Therefore, $P$, and so, $I$ is a subset of $P_\alpha$, for some $\alpha$.
	
	$(1) \Rightarrow (5)$: Let $S$ be a compactly packed semiring and $I$ a radical ideal of $S$, i.e., $I = \sqrt{I}$. Set \[ D(I) = \Spec(S) \setminus V(I).\] Clearly, $I$ is not a subset of any prime ideal in $D(I)$. Since $S$ is compactly packed, $I \nsubseteq \bigcup D(I)$. This means that there is an element $x \in I$ which is not in $\bigcup D(I)$. This implies that $V(x) \subseteq V(I)$ because if $P$ is a prime ideal in $V(x)$, then $P$ is not among the prime ideals with $P \nsupseteq I$. This means that $P$ is a prime ideal with $P \supseteq I$. On the other hand, $x \in I$ is equivalent to $(x) \subseteq I$ which implies that $V(I) \subseteq V(x)$. Consequently, $V(x) = V(I)$. In view of Theorem 3.2 in \cite{Nasehpour2018P}, this implies that \[\sqrt{(x)} = \bigcap V(x) = \bigcap V(I) = \sqrt{I} = I.\] Note that $(5)$ implies $(4)$, and $(4)$ implies $(3)$. This completes the proof. 
\end{proof}

\begin{remark}
A ring version of the definition of compactly packed semirings given in Definition \ref{compactlypackedsemirings} is due to C.M. Reis and T.M. Viswanathan \cite{ReisViswanathan1970} who proved a version of Theorem \ref{generalizedpal} for Noetherian rings. The ring version of Theorem \ref{generalizedpal} is due to William Walker Smith (see the main theorem of the paper \cite{Smith1971}), and Pakala and Shores (see Theorem 1 in \cite{PakalaShores1981}).	
\end{remark}

\begin{proposition}\label{examplecompactlypacked2}
Let $(L,+,\cdot)$ be a bounded distributive lattice with finitely many prime ideals. Then, $L$ is a compactly packed semiring.
\end{proposition}

\begin{proof}
In view of Theorem 3 and Theorem 9 in \cite{Nasehpour2016}, each ideal of $L$ is subtractive. In view of Example \ref{examplescompactlypacked}, $L$ is compactly packed because by assumption, $L$ has finitely many prime ideals. This completes the proof.
\end{proof}

\begin{proposition}\label{examplecompactlypacked3}
Let $T = [0,1]$ be a finite chain with the smallest element $0$ and the largest element $1$. Define addition on $T$ as $a+b = \max\{a,b\}$ and multiplication to be null except for the case $a \cdot 1 = 1 \cdot a = a$, for all $a$ and $b$ in $T$. Then, $T$ is a compactly packed semiring.
\end{proposition}

\begin{proof}
By Proposition 21 in \cite{Nasehpour2016}, each prime ideal of the commutative semiring $T$ is subtractive. Clearly, $T$ is compactly packed because $T$ is a finite semiring and has finitely many prime ideals. This completes the proof. 
\end{proof}

The topological interpretation of the covering condition for the compactly packed semirings is the following:

\begin{theorem}
	Let $S$ be a compactly packed semiring. In the Zariski topology on $X = \Spec(S)$, if an arbitrary number of points are contained in an open subset then they are contained in a smaller principal open subset.
\end{theorem}

\begin{proof}
In view of the Definition \ref{compactlypackedsemirings} and the proof given for Theorem \ref{generalizedpal}, the proof is essentially the same as the one given in Theorem \ref{primeavoidancetopology}, and so, omitted.
\end{proof}

\section{Zero-divisors of semimodules over commutative semirings}\label{sec:zerodivisors}

\begin{proposition}\label{Mannideals}
Let $M$ be a semimodule over a commutative semiring $S$. Then, the following statements hold:

\begin{enumerate}
	\item For any nonempty subset $X$ of $M$, $\Ann(X)$ is a subtractive ideal of $S$.
	
	\item For a family of nonempty subsets $\{X_\alpha\}_\alpha$ of $M$, we have \[\bigcap_\alpha \Ann(X_\alpha) = \Ann(\cup_\alpha X_\alpha).\]
\end{enumerate}
\end{proposition}

\begin{proof}
The proofs are similar to the proof of Theorem \ref{annihilatoridealssubtractive} and Proposition \ref{intersectionannihilatorideals}, and so, omitted.
\end{proof} 

\begin{proposition}\label{zerodivisorssemimoduleunionradicalideals}
	Let $S$ be a commutative semiring and $M$ an $S$-semimodule. Then, the set of zero-divisors $Z(M)$ of $M$ is a union of radicals of $M$-annihilator ideals of $S$; more precisely, \[Z(M) = \bigcup_{0 \neq x \in M} \sqrt{\Ann(x)}.\]
\end{proposition}

\begin{proof}
It is clear that \[Z(M) = \bigcup_{0 \neq x \in M} \Ann(x) \subseteq  \bigcup_{0 \neq x \in M} \sqrt{\Ann(x)}.\] Now, let $s \in \sqrt{\Ann(x)}$, for some nonzero $x$ in $M$. This implies that $s^n x = 0$, for some positive integer $n$. If $s$ is not a zero-divisor on $M$, then $x = 0$, a contradiction. Thus $s \in Z(M)$ and the proof is complete.
\end{proof}

\begin{remark}
Proposition \ref{zerodivisorssemimoduleunionradicalideals}, which is a generalization of Proposition 1.15 in \cite{AtiyahMacdonald2016}, states that the set of zero-divisors of a semimodule is a union of radical ideals. This condition that the set of zero-divisors of a module over a commutative ring is a union of prime ideals is of special interests \cite[Theorem 6.1]{Matsumura1989}. We investigate this condition in the context of semimodule theory.
\end{remark}

The following is a generalization of Corollary 2.9 in \cite{Nasehpour2021}.

\begin{theorem}\label{accdccannidealsvfzd}
Let $M$ be a semimodule over a commutative semiring $S$. If $S$ has a.c.c. and d.c.c. on its $M$-annihilator ideals, then $M$ has very few zero-divisors.
\end{theorem}

\begin{proof}
Since $S$ has a.c.c. on its $M$-annihilator ideals, the maximal elements of the set \[\mathcal{A} = \{\Ann(x) : x \in M \setminus \{0\}\}\] exist and are prime ideals of $S$ (Theorem 2.8 in \cite{Nasehpour2021}). Let $P_i = \Ann(x_i)$ be the maximal elements of $\mathcal{A}$. Our claim is that the number of such $P_i$s is finite. On the contrary, let $\{P_i\}_{i=1}^{+\infty}$ be among the maximal elements of $\mathcal{A}$. Since $S$ has d.c.c. on its $M$-annihilator ideals, the descending chain \[P_1 \supseteq P_1 \cap P_2 \supseteq \cdots \] of $M$-annihilator ideals of $S$ terminates at some point, say at $r$. This implies that \[P_1 \cap \dots \cap P_r = P_1 \cap \dots \cap P_r \cap P_{r+1} \subseteq P_{r+1}.\] Since $P_{r+1}$ is prime, by Proposition \ref{multiplicationidealsubsetintersection} and Theorem \ref{primenesscriterionsemiringthm}, $P_i \subseteq P_{r+1}$ for some $i \leq r$, contradicting the maximality of $P_i$s. Thus $Z(M) = \bigcup_{i=1}^{r} P_i$ and the proof is complete. 
\end{proof}

The following is a generalization of Theorem 49 in \cite{Nasehpour2016}.

\begin{corollary}\label{accannidealsvfzd}
Let $S$ be a commutative semiring. If $S$ has a.c.c. on its annihilator ideals, then $S$ has very few zero-divisors.
\end{corollary}

\begin{proof}
By Corollary 1.4 in \cite{BehzadipourNasehpour2023}, if $S$ has a.c.c. on its annihilator ideals, then $S$ has d.c.c. on its annihilator ideals. Thus by Theorem \ref{accdccannidealsvfzd}, $S$ has very few zero-divisors and the proof is complete. 
\end{proof}

Let $R$ be a Noetherian commutative ring with identity and $M$ a Noetherian unital $R$-module. If an ideal $I$ is a subset of $Z(M)$, then $I \subseteq P$, for some $P \in \Ass(M)$ (\cite[Theorem 82]{Kaplansky1974}). Kaplansky on p. 56 of his book \cite{Kaplansky1974} describes this result ``among the most useful in the theory of commutative rings''. As an application of prime avoidance for semirings, we generalize this in the context of semimodule theory as follows:

\begin{theorem}\label{amongmostusefulkaplansky}
Let $I$ be an ideal of a commutative semiring $S$ and $M$ an $S$-semimodule. Also, let $S$ have a.c.c. and d.c.c. on its $M$-annihilator ideals. If $I \subseteq Z(M)$, then $I \subseteq P$, for some $P \in \Ass(M)$.
\end{theorem}

\begin{proof}
If $I \subseteq Z(M)$, then by Theorem \ref{PALforsemirings} (also see \S7 in \cite{Nasehpour2016}) $I \subseteq P$, for some $P \in \Ass(M)$ and the proof is complete. 
\end{proof}

\begin{remark}
	Theorem \ref{amongmostusefulkaplansky} provides numerous examples of proper semimodules having Property (A). Recall that an $S$-semimodule $M$ has Property (A) if each finitely generated ideal $I \subseteq Z(M)$ has a nonzero annihilator in $M$ \cite[Definition 2.14]{Nasehpour2021}. Also, note that we define a semimodule to be proper if it is not a module. 	
\end{remark}

Davis defines a commutative ring $R$ to have few zero-divisors if $Z(R)$ is a finite union of prime ideals. Next, he proves that a commutative ring $R$ has few zero-divisors if and only if its total quotient ring $Q(R)$ is semi-local (see pages 203 and 204 in \cite{Davis1964}).

\begin{definition}\label{semiringfewzerodivisors}
	We say a commutative semiring $S$ has few zero-divisors if $Z(S)$ is a finite union of subtractive prime ideals of $S$.
\end{definition}

\begin{remark}
	Since any prime ideal in $\Ass(S)$ is of the form of $\Ann(x)$, for some $x \in S$, it is subtractive. Therefore, if a commutative semiring has very few zero-divisors, then it needs to have few zero-divisors! It follows that if a commutative semiring $S$ has a.c.c. on its annihilator ideals, then it has (very) few zero-divisors (Corollary \ref{accannidealsvfzd}).
\end{remark}

Localization of a commutative semiring $S$ at a multiplicative closed set $U$ of $S$ is defined similarly to its counterparts in commutative ring theory \cite[\S5]{Nasehpour2018S}. Observe that if $S$ is a commutative semiring, then $S \setminus Z(S)$ is multiplicatively closed in $S$. The localization of $S$ at $S \setminus Z(S)$ is called to be the total quotient semiring of $S$, denoted by $Q(S)$. Note that a semiring is semi-local if it has finitely many maximal ideals \cite[Definition 3.14]{Nasehpour2018S}.

The following is an extension of Davis' result on rings having few zero-divisors and another example of an application of prime avoidance in semiring theory:

\begin{theorem}\label{fzdtquotientsemiringsemilocal}
	Let a commutative semiring $S$ have few zero-divisors. Then, its total quotient semiring $Q(S)$ is semi-local.
\end{theorem}

\begin{proof}
	Let $S$ be a commutative semiring. By Theorem 5.4 in \cite{Nasehpour2018S}, the prime ideals of $Q(S)$ is in a one-to-one correspondence with the prime ideals of $S$ disjoint from $S \setminus Z(S)$. Consequently, any prime ideal of $Q(S)$ corresponds to a prime ideal $P$ of $S$ with $P \subseteq Z(S)$. By definition, since $S$ has few zero-divisors, \[Z(S) = P_1 \cup P_2 \cup \dots \cup P_n,\] where $P_i$s are subtractive prime ideals of $S$. By prime avoidance (Theorem \ref{PALforsemirings}), we have $P \subseteq P_i$, for some $i \in \mathbb{N}_n$. Therefore, each prime ideal of $Q(S)$ is a subset of a prime ideal of $Q(S)$ obtained by the extension of $P_i$ in $Q(S)$ for some $i$. This means that the only maximal ideals of $Q(S)$ are the extended ideals of $P_i$s in $Q(S)$. Thus $Q(S)$ is semi-local and the proof is complete.    
\end{proof}

\begin{theorem}\label{extensionannihilatorideals}
If $I= \Ann(x)$ for some element $x$ in a commutative semiring $S$ and $U \subseteq S \setminus Z(S)$ is multiplicatively closed in $S$, then the localization of $I$ at $U$ is $I_U = \Ann(x/1)$ for the element $x/1 \in S_U$.
\end{theorem}

\begin{proof}
	$\subseteq$: Let $s/u \in I_U$, where $s \in I$ and $u \in U$. From $sx = 0$, it follows that \[s/u \cdot x/1 = sx / u = 0 / u = 0.\] So, $s/u \in \Ann(x/1)$. 
	
	$\supseteq$: Now, let $s/u \in \Ann(x/1)$. So, $(s/u)(x/1) = 0/1$. This means that $(sx)/u = 0/1$. Consequently, there is an element $t \in U$ such that $tsx = 0$. Since $t$ is not a zero-divisor on $S$, we have $sx = 0$ and the proof is complete.   
\end{proof}

A commutative ring $R$ is a Kasch ring if and only if each maximal ideal of $R$ is of the form $\Ann(x)$, for some $x \in R$ (see Definition 8.26 and Corollary 8.28 in \cite{Lam1999}). Kasch rings were named after Friedrich Kasch who was a student of Friedrich Karl Schmidt. Inspired by this, we give the following definition:

\begin{definition}\label{Kaschsemiringdef}
	We say a commutative semiring $S$ is a Kasch semiring if each maximal ideal of $S$ is of the form $\Ann(x)$, for some $x \in S$. 
\end{definition}

\begin{corollary}\label{acctqsemilocalKasch}
	Let a commutative semiring $S$ have a.c.c. on its annihilator ideals. Then, $Q(S)$ is a semi-local Kasch semiring having very few zero-divisors.
\end{corollary}

\begin{proof}
	Since $S$ has a.c.c. on its annihilator ideals, $S$ has very few zero-divisors (Corollary \ref{accannidealsvfzd}). This means that $Z(S)$ is a finite union of prime ideals in $\Ass(S)$. Therefore by Theorem \ref{fzdtquotientsemiringsemilocal}, $Q(S)$ is semi-local. On the other hand, by Theorem \ref{extensionannihilatorideals}, each maximal ideal of $Q(S)$ is an annihilator of an element in $Q(S)$. Now, since $Z(Q(S)) = \bigcup_{0 \neq x \in Q(S)} \Ann(x)$, it follows that $Q(S)$ has very few zero-divisors and the proof is complete.
\end{proof}

\begin{remark}
	In Theorem 8.31 in \cite{Lam1999}, Lam attributes a ring version of Corollary \ref{acctqsemilocalKasch} to Carl Faith.
\end{remark}

Let $S$ be a commutative semiring and $G$ a commutative monoid. Note that if $f$ is an element of a monoid semiring $S[G]$, its content $c(f)$ is defined to be an ideal of $S$ generated by the coefficients of $f$. In other words, if \[f= s_0 + s_1 X^{g_1} + \dots + s_n X^{g_n},\qquad s_i \in S,~g_i \in G\] is the canonical representation \cite[p. 68]{Gilmer1984} of an element $f$ in the monoid semiring $S[G]$, then the content of $f$ is: \[c(f) = (s_0, s_1, \dots, s_n).\]

\begin{theorem}\label{zerodivisorsmonoidsemimodule1}
Let $S$ be a commutative semiring and $M$ an $S$-semimodule with Property (A) and $Z(M) = \bigcup_{\alpha \in A} P_\alpha$, where $P_\alpha$s are prime ideals of $S$. If $G$ is a cancellative torsion-free commutative monoid and $S$ is compactly packed, then $Z(M[G]) =  \bigcup_{\alpha \in A} P_\alpha[G]$. 	
\end{theorem}

\begin{proof}
$(\subseteq)$: If $f$ is a zero-divisor on the monoid semimodule $M[G]$, then by Theorem 2.1 in \cite{Nasehpour2021}, there is a nonzero element $b \in M$ such that $fb = 0$. This means that \[c(f) \subseteq Z(M) = \bigcup_{\alpha \in A} P_\alpha.\] Since $S$ is compactly packed, the ideal $c(f)$ needs to be a subset of one of the prime ideals $P_\alpha$s. This implies $f$ to be an element of $P_\alpha [G]$. 

$(\supseteq)$: On the other hand, if $f \in P_\alpha[G]$, then $c(f)$ is a subset of $P_\alpha$. This implies that $c(f) \subseteq Z(M)$. Now, since $M$ has Property (A), $c(f)$ can be annihilated by a nonzero element $b \in M$. Hence, $f$ is a zero-divisor of on $M[G]$, as required.
\end{proof}

\begin{theorem}\label{zerodivisorsmonoidsemimodule2}
Let $S$ be a Noetherian commutative semiring and $M$ an $S$-semimodule. If $S$ is compactly packed and $G$ is a cancellative torsion-free commutative monoid, then $Z(M[G]) = \bigcup_{P \in \Ass(M)} P[G]$.
\end{theorem}

\begin{proof}
By Corollary 2.9 in \cite{Nasehpour2021}, $Z(M) = \bigcup_{P \in \Ass(M)} P$.

$(\subseteq)$: Let $f \in Z(M[G])$. By Theorem 2.1 in \cite{Nasehpour2021}, $f \cdot b = 0$, for some $b \in M$. This implies that $c(f) \cdot b = 0$. Therefore, \[c(f) \subseteq Z(M) = \bigcup_{P \in \Ass(M)} P.\] Since $S$ is compactly packed, we have $c(f) \subseteq P$, for some $P \in \Ass(M)$. Thus $f \in P[G]$, for some $P \in \Ass(M)$.

$(\supseteq)$: If $f \in P[G]$, for some $P \in \Ass(M)$, $c(f) \subseteq P$, where there is a $b \in M$ with $P = \Ann(b)$. This implies that $f$ is annihilated by $b$, i.e., $f \in Z(M[G])$ and the proof is complete.
\end{proof}

\section*{Acknowledgments and Dedicatory}

Prof. Dr. J\"{u}rgen Herzog (1941--2024) was a prominent mathematician who, together with Prof. Dr. Winfried Bruns, co-authored one of the most influential textbooks on commutative algebra \cite{BrunsHerzog1998}. He was also a kind soul, helping young mathematicians find their way in mathematics at the research level. This paper is dedicated to his memory. The author is grateful to the anonymous referees for their thorough reading of the paper and their useful comments, which improved its presentation. 

\bibliographystyle{plain}

\end{document}